\documentclass[11pt]{amsart}
\usepackage[left=1.25in,right=1.25in]{geometry}
\usepackage{amsthm}
\usepackage{amsmath}
\usepackage{mathtools}
\usepackage{amssymb}
\usepackage{hyperref}
\usepackage{cleveref}
\usepackage{tikz-cd}
\usepackage{tikz}
\usepackage{array}
\usetikzlibrary{decorations.markings}
\usepackage{enumerate}
\usepackage{amsfonts}
\usepackage{todonotes}
\usepackage[utf8]{inputenc}
\usepackage{ytableau}
\usepackage{subcaption}
\usepackage[export]{adjustbox}
\usepackage{mathdots}
\usepackage{comment}
\usepackage{rotating}
\usetikzlibrary{shapes}

\newtheorem{theorem}{Theorem}[section]

\newtheorem{lemma}[theorem]{Lemma}

\theoremstyle{definition}
\newtheorem{ex}[theorem]{Example}

\newtheorem{defin}[theorem]{Definition}
\newtheorem{remark}[theorem]{Remark}

\DeclareTextFontCommand{\emph}{\color{blue}\em}

\title[Characteristic polynomials of deformations of Coxeter arrangements]{Characteristic Polynomials of Deformations of Coxeter Arrangements via levels of regions}
\author{Ningxin Zhang}
\address{School of Mathematical Sciences, Peking University}
\email{\href{mailto:xinxin9909@pku.edu.cn}{{\tt xinxin9909@pku.edu.cn}}}

\date{\today}

\begin{document}
\begin{abstract}
    We obtain a novel formula for characteristic polynomials of deformations of the Braid arrangement using the notion of levels of regions. As an application, we recover and strengthen results of Chen et al. \cite{wsj,wsj2} on the characteristic polynomial of several specific types of hyperplane arrangements via much simpler arguments. Our theorem also generalizes to type $B$.
\end{abstract}
\maketitle

\section{Introduction}
The study of Coxeter arrangements plays an essential role in the theory of hyperplane arrangements \cite{terao,stanley2004introduction}, largely because of their significant connections to algebra, particularly in the context of reflection groups and invariant theory. Coxeter arrangements arise from the reflecting hyperplanes associated with the root systems of finite Coxeter groups, which reveal the symmetries and combinatorial structures in a geometric context.

Deformations of Coxeter arrangements are affine arrangements where each hyperplane is parallel to some hyperplane of the original Coxeter arrangement. Numerous special examples of these arrangements have been extensively studied over the years \cite{athanasiadis00,POSTNIKOV2000544}, including the Catalan arrangement and the Shi arrangement, particularly concerning characteristic polynomials and the enumeration of regions. 

In this paper, we establish a formula for the characteristic polynomial of general deformations of Coxeter arrangements, expanding the polynomial into terms related to the numbers of regions with different levels. The level of a region, defined as the dimension of unbounded directions or the degree of freedom within the region, was first introduced by Ehrenborg in 2019 \cite{EHRENBORG201955}. We will present the precise definition in \Cref{sec:back}.

The \emph{Coxeter arrangement} of type $A_{n-1}$ in $\mathbb{R}^n$ is 
\[\operatorname{Cox}_{\mathcal{A}}(n) = \{x_i-x_j = 0\ |\ 1 \le i \ne j \le n\}.\]
We define the \emph{non-degenerate deformation} of the Coxeter arrangement of type $A_{n-1}$
\begin{equation}\label{typeA}
\mathcal{A} = \{x_i-x_j = a_{ij}^{(1)},\ldots,a_{ij}^{(t_{ij})}\ |\ 1\le i \ne j \le n\},
\end{equation}
where $a_{ij}^{(1)},\ldots,a_{ij}^{(t_{ij})} \in \mathbb{R}$ and $t_{ij} \ge 1$ for all $1 \le i,j \le n$. Notice that in the original definition of deformations, the number of hyperplanes in each direction $t_{ij}$ can be any nonnegative integer. However, in this article we focus on the \textbf{non-degenerate} deformations, meaning that in each direction there is at least one hyperplane within the arrangement. 

The following is our main theorem, a new expansion of the characteristic polynomial of deformations of the type $A$ Coxeter arrangements.
\begin{theorem}\label{thm:typeA}
    Let $\mathcal{A}$ be a non-degenerate deformation of $\operatorname{Cox}_{A}(n)$ as in \Cref{typeA}. Then
    \[\chi_{\mathcal{A}}(t) = \sum_{k = 0}^{n} (-1)^{n-k} \cdot r_k(\mathcal{A}) \cdot \binom{t}{k},\]
    where $r_k(\mathcal{A})$ is the number of regions with level $k$ in arrangement $\mathcal{A}$.
\end{theorem}
\begin{remark}
    When $t = -1$, the right hand side of the formula
    \[\text{RHS} = \sum_{k = 0}^n (-1)^{n-k} \cdot r_k(\mathcal{A}) \cdot (-1)^{k} = (-1)^n \cdot r(\mathcal{A}),\]
    which is consistent with Zaslavsky's theorem (see \Cref{thm:zas}).
\end{remark}
Our result \Cref{thm:typeA} generalizes several recent results by Chen et al. (see Theorem 1.5 of \cite{wsj} and Theorem 1.2 of \cite{wsj2}) on the characteristic polynomials of a specific type of arrangements via much simpler and more general arguments.

\begin{ex}
    Let $\mathcal{A} = \{H_1,H_2,H_3,H_4,H_5\}$ be a hyperplane arrangement in $\mathbb{R}^3$, where \[H_1: x_1-x_2 = 0, \ H_2: x_1-x_2 = 1, \ H_3: x_2-x_3 = 0, \ H_4: x_1-x_3 = 1, \ H_5: x_1-x_3 = 0.\]
    \Cref{subfig:A} shows the projection of the arrangement $\mathcal{A}$ onto the plane $x _1 + x_2 + x_3 = 0$, where all the regions are labeled by their levels. The characteristic polynomial
    \[\chi_{\mathcal{A}}(t) = t^3 -5t^2 + 6t = 6\binom{t}{3} - 4\binom{t}{2} + 2 \binom{t}{1},\]
    where $r_3(\mathcal{A}) = 6$, $r_2(\mathcal{A}) = 4$, and $r_1(\mathcal{A}) = 2$.
\end{ex}

\Cref{thm:typeA} can be extended to type $B$ deformations as well. The Coxeter arrangement of type $B_n$ in $\mathbb{R}^n$ is
\[\operatorname{Cox}_{\mathcal{B}}(n) = \{x_i = 0\ |\ 1 \le i \le n\} \cup \{x_i\pm x_j = 0\ |\ 1 \le i,j \le n\}.\]
Similarly, we define the \emph{non-degenerate deformation} of the Coxeter arrangement of type $B_n$
\begin{equation}\label{typeB}
\begin{aligned}
\mathcal{B} = & \{x_i = a_i^{(1)},\ldots,a_i^{(r_i)} \ |\ 1 \le i \le n\} \\
& \cup \{x_i-x_j = b_{ij}^{(1)},\ldots,b_{ij}^{(s_{ij})}\ |\ 1\le i \ne j \le n\} \\
& \cup \{x_i+x_j = c_{ij}^{(1)},\ldots,c_{ij}^{(t_{ij})}\ |\ 1\le i \ne j \le n\},
\end{aligned}
\end{equation}
where $a_i^{(1)},\ldots,a_i^{(r_i)},b_{ij}^{(1)},\ldots,b_{ij}^{(s_{ij})},c_{ij}^{(1)},\ldots,c_{ij}^{(t_{ij})} \in \mathbb{R}$ and $r_i,s_{ij,}t_{ij} \ge 1$ for all $1 \le i,j \le n$. As in the type $A$ case, the non-degenerate deformation of type $B$ still requires that there is at least one hyperplane in each direction of the arrangement. 

Analogous to \Cref{thm:typeA}, we have the following main theorem, an expansion of the characteristic polynomial of deformations of the type $B$ Coxeter arrangements.  

\begin{theorem}\label{thm:typeB}
    Let $\mathcal{B}$ be a non-degenerate deformation of $\operatorname{Cox}_{B}(n)$ as in \Cref{typeB}. Then,  
    \[\chi_{\mathcal{B}}(t) = \sum_{k = 0}^{n} (-1)^{n-k}\cdot r_k(\mathcal{B}) \cdot \binom{\frac{t-1}{2}}{k},\]
    where $r_k(\mathcal{B})$ is the number of regions with level $k$ in arrangement $\mathcal{B}$.
\end{theorem}
\begin{remark}
    When $t = -1$, the right hand side of the formula
    \[\text{RHS} = \sum_{k = 0}^n (-1)^{n-k} \cdot r_k(\mathcal{B}) \cdot (-1)^k = (-1)^n \cdot r(\mathcal{B}),\]
    which is consistent with Zaslavsky's theorem.
\end{remark}

\begin{ex}
    Let $\mathcal{B} = \{H_1,H_2,H_3,H_4,H_5\}$ be a hyperplane arrangement in $\mathbb{R}^3$ shown in \Cref{subfig:B}, where \[H_1: x_1 = 0, \ H_2: x_1-x_2 = 0, \ H_3: x_2 = 0, \ H_4: x_1+x_2 = 1.\]
    The characteristic polynomial
    \[\chi_{\mathcal{B}}(t) = t^2 -4t + 5 = 8\binom{\frac{t-1}{2}}{2} + 2 \binom{\frac{t-1}{2}}{0},\]
    where $r_2(\mathcal{B}) = 8$, $r_1(\mathcal{B}) = 0$, and $r_0(\mathcal{B}) = 2$.
\end{ex}
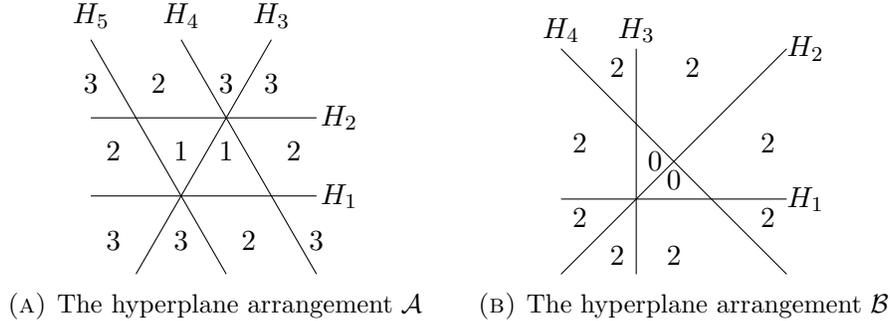
\begin{figure}[h]
    \centering
    \begin{subfigure}{.4\textwidth}
    \centering
    \begin{tikzpicture}[scale = 0.6]
        \draw (-1,0)--(4,0);
        \draw (-1,{sqrt(3)})--(4,{sqrt(3)});
        \draw (0,{-sqrt(3)})--(3,{2*sqrt(3)});
        \draw (1,{2*sqrt(3)})--(4,{-sqrt(3)});
        \draw (-1,{2*sqrt(3)})--(2,{-sqrt(3)});
        \node at (3,4) {$H_3$};
        \node at (1,4) {$H_4$};
        \node at (-1,4) {$H_5$};
        \node at (4.5,{sqrt(3)}) {$H_2$};
        \node at (4.5,0) {$H_1$};
        \node at (1,1) {$1$};\node at (2,1) {$1$};
        \node at (-0.5,1) {$2$};\node at (3.5,1) {$2$};\node at (0.5,2.5) {$2$};\node at (2.5,-1) {$2$};
        \node at (-1,2.5) {$3$};\node at (2,2.5) {$3$};\node at (3,2.5) {$3$};\node at (-0.5,-1) {$3$};\node at (1,-1) {$3$};\node at (4,-1) {$3$};
    \end{tikzpicture}
    \caption{The hyperplane arrangement $\mathcal{A}$}
    \label{subfig:A}
    \end{subfigure}
    \begin{subfigure}{.4\textwidth}
    \centering
    \begin{tikzpicture}[scale = 0.5]
        \draw (-2,2)--(4,2);
        \draw (0,0)--(0,6);
        \draw (-2,0)--(4,6);
        \draw (-2,6)--(4,0);
        \node at (0,6.5) {$H_3$};
        \node at (-2,6.5) {$H_4$};
        \node at (4.5,6) {$H_2$};
        \node at (4.5,2) {$H_1$};
        \node at (0.5,3) {$0$};\node at (1,2.5) {$0$};
        \node at (-1.5,3.5) {$2$};\node at (-0.5,5.5) {$2$};\node at (3.5,3.5) {$2$};\node at (1.5,5.5) {$2$};\node at (3.5,1.5) {$2$};\node at (1,0.5) {$2$};\node at (-0.5,0.5) {$2$};\node at (-1.5,1.5) {$2$};
    \end{tikzpicture}
    \caption{The hyperplane arrangement $\mathcal{B}$}
    \label{subfig:B}
    \end{subfigure}
    \caption{Examples of deformations of Coxeter arrangements.}
\end{figure}
The outline of the paper is as follows. In \Cref{sec:back} we give necessary background on hyperplane arrangements. In \Cref{sec:proof} we prove our main results, which are \Cref{thm:typeA} and \Cref{thm:typeB}, and give some applications as well.

\section{Background}\label{sec:back}
\subsection{Characteristic polynomials}
A \emph{hyperplane arrangement} $\mathcal{A} = \{H_1,\ldots,H_m\}$ is a finite set of affine hyperplanes in $\mathbb{R}^n$. The \emph{intersection poset} $L(\mathcal{A})$ of arrangement $\mathcal{A}$ is the set of all nonempty intersections of hyperplanes in $\mathcal{A}$, including $\mathbb{R}^n$ itself, partially ordered by reverse inclusion.

\begin{defin}
    The \emph{characteristic polynomial} $\chi_{\mathcal{A}}(t)$ of the arrangement $\mathcal{A}$ is defined by 
    \[\chi_{\mathcal{A}}(t) = \sum_{x \in L(\mathcal{A})} \mu(x) t^{\dim(x)},\]
    where $L(\mathcal{A})$ is the intersecting poset of $\mathcal{A}$, and $\mu(x) = \mu(\hat{0},x)$ is the M\"{o}bius function of $L(\mathcal{A})$. 
\end{defin}

Let $\mathcal{A}$ be an arrangement in $\mathbb{R}^n$. Given a hyperplane $H_0 \in \mathcal{A}$, define the \emph{restriction arrangement} $\mathcal{A}^{H_0}$ in the affine subspace $H_0 \cong \mathbb{R}^{n-1}$ by
\[\mathcal{A}^{H_0} = \{H_0 \cap H \neq \emptyset: \ H \in \mathcal{A} {-} \{H_0\}\}.\]
Let $\mathcal{A}^{\prime} = \mathcal{A} {-} \{H_0\}$ and $\mathcal{A}^{\prime\prime} = \mathcal{A}^{H_0}$. We call $(\mathcal{A},\mathcal{A}^{\prime},\mathcal{A}^{\prime\prime})$ a \emph{triple of arrangements with distinguished hyperplane} $H_0$.
The characteristic polynomial has a fundamental recursive property. 

\begin{lemma}[Deletion-restriction \cite{stanley2004introduction}]\label{lemma:d-r}
    Let $(\mathcal{A},\mathcal{A}^{\prime},\mathcal{A}^{\prime\prime})$ be a triple of arrangements. Then
    \[\chi_{\mathcal{A}}(t) = \chi_{\mathcal{A}^{\prime}}(t) - \chi_{\mathcal{A}^{\prime\prime}}(t).\]
\end{lemma}

We call $\mathcal{A}$ an \emph{integer-arrangement} if all of its hyperplanes are given by equations with integer coefficients. For such an arrangement $\mathcal{A}$, the following well-known result shows that the characteristic polynomial can be computed by counting the cardinality of certain finite fields.
\begin{theorem}[\cite{athanasiadis96}]\label{thm:terao}
    Let $\mathcal{A}$ be an rational arrangement in $\mathbb{R}^n$. Given a sufficiently large prime power $q$, then $\chi_{\mathcal{A}}(q)$ is equal to the number of points in $\mathbb{F}_q^n$ that do not belong to any of the hyperplanes in arrangement $\mathcal{A}$. 
\end{theorem}

\subsection{Levels of regions}
A \emph{region} of an arrangement $\mathcal{A}$ is a connected component of the complement of the hyperplanes. Let $\mathcal{R}(\mathcal{A})$ denote the set of regions of $\mathcal{A}$, and let
\[r(\mathcal{A}) = \left|\mathcal{R}(\mathcal{A})\right|\]
denote the number of regions in the arrangement $\mathcal{A}$.

\begin{theorem}[Zaslavsky Theorem]\label{thm:zas}
     Let $\mathcal{A}$ be an arrangement in $\mathbb{R}^n$. Then
     \[r(\mathcal{A}) = (-1)^n \chi_{\mathcal{A}}(-1).\]
\end{theorem}

The following definition may not be as familiar to the audience.
\begin{defin}
    Given a subset $X \subset \mathbb{R}^n$, the \emph{level} of $X$ is the smallest non-negative integer $\ell$ such that
    \[ X \subset B(W,r) = \{x \in \mathbb{R}^n: d(x,W) \le r\},\]
    for some subspace $W$ of dimension $\ell$ and a real number $r > 0$.
\end{defin}
Informally speaking, the level of a region equals the dimension of its unbounded directions, reflecting the region's degree of freedom. 

Let $\mathcal{R}_{\ell}(\mathcal{A})$ denote the collection of regions of $\mathcal{A}$ with level $\ell$, and let
\[r_{\ell}(\mathcal{A}) = \left| \mathcal{R}_{\ell}(\mathcal{A})\right|.\]

\begin{ex}
Let $\mathcal{A} = \{H_1,H_2,H_3,H_4\}$ be a hyperplane arrangement in $\mathbb{R}^2$, see \Cref{fig:level}, where \[H_1: x = 0, \ H_2: y = 0, \ H_3: x+y = 1, \ H_4: y = 1.\]
The characteristic polynomial of $\mathcal{A}$ is $\chi_{\mathcal{A}}(t) = t^2-4t+4$. 
For the three regions labeled in \Cref{fig:level}, we show that $\ell(\Delta_0) = 0$, $\ell(\Delta_1) = 1$, $\ell(\Delta_2) = 2$. We have the number of regions with each level $r_0(\mathcal{A}) = 1$, $r_1(\mathcal{A}) = 2$, $r_2(\mathcal{A}) = 6$, and $r(\mathcal{A}) = r_0(\mathcal{A}) + r_1(\mathcal{A}) + r_2(\mathcal{A}) = 9$.
\end{ex}
\begin{figure}[h]
    \centering
    \begin{tikzpicture}[scale = 0.6]
        \draw (0,0)--(6,0);
        \draw (0,2)--(6,2);
        \draw (2,-2)--(2,4);
        \draw (0,4)--(6,-2);
        \node at (2.5,1) {$\Delta_0$};
        \node at (4,1) {$\Delta_1$};
        \node at (3,3) {$\Delta_2$};
        \node at (2,4.5) {$H_1$};
        \node at (0,4.5) {$H_3$};
        \node at (6.5,2) {$H_4$};
        \node at (6.5,0) {$H_2$};
    \end{tikzpicture}
    \caption{The hyperplane arrangement $\mathcal{A}$}
    \label{fig:level}
\end{figure}
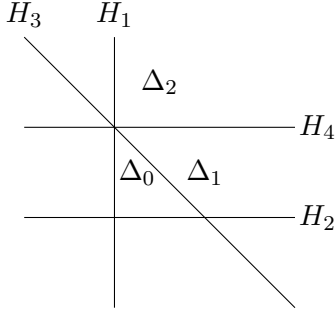

\section{Expanding $\chi(t)$ to the numbers of regions with different levels}\label{sec:proof}
In this section, we proof our main results \Cref{thm:typeA} and \Cref{thm:typeB}. The method involves applying the deletion-restriction lemma (see \Cref{lemma:d-r}) and using induction on the dimension of the arrangements. The crucial part is to show that the restriction of a deformation on a hyperplane is also an arrangement of desired form. This allows us to utilize the recursive property of characteristic polynomials. 

\subsection{Deformations of type $A$ Coxeter arrangements}
Before proving \Cref{thm:typeA}, we establish the following lemma to address the key component of the proof.
\begin{lemma}\label{lemma:nonde}
    Let $\mathcal{A}$ be a non-degenerate deformation of type $A$ Coxeter arrangement in $\mathbb{R}^n$. Choose a hyperplane $H_0 \in \mathcal{A}$. Then the restriction of $\mathcal{A}$ on $H_0$ is a non-degenerate deformation of type $A$ Coxeter arrangement in $\mathbb{R}^{n-1}$.
\end{lemma}

\begin{proof}
    Suppose that $H_0: x_k - x_l = a$ for some $1 \le k < l \le n$ and $a \in \mathbb{R}$. 
    Let $\pi$ be a projection from $H_0$ to $\mathbb{R}^{n-1}$ by 
    \[\pi (x_1, \dots,x_l, \ldots x_n) = (x_1, \ldots,\hat{x}_l, \ldots,x_n).\]
    It is not hard to tell that $\pi$ is an isomorphism. Therefore, the restriction arrangement $\mathcal{A}^{H_0}$ is isomorphic to the arrangement $\pi(\mathcal{A}^{H_0})$ in $\mathbb{R}^{n-1}$ under the projection $\pi$. For any $H \in \mathcal{A}-{H_0}$ such that $H \cap H_0 \ne \emptyset$, $H$ has the form
    \[H: x_i - x_j = b\]
    for some $1 \le i<j \le n$ and $b \in \mathbb{R}$. We obtain all the hyperplanes in arrangement $\pi(\mathcal{A}^{H_0})$ as follows by taking the projection of nonempty intersections with $H_0$.
    \[\pi(H \cap H_0) = \begin{cases}
        \{x \in \mathbb{R}^{n-1}: \ x_k - x_j = a+b \}, &\text{if}\ i = l, \\
        \{x \in \mathbb{R}^{n-1}: \ x_k - x_i = a-b \}, &\text{if}\ j = l, \\
        \{x \in \mathbb{R}^{n-1}: \ x_i - x_j = b\}, &\text{otherwise}.\\
    \end{cases}\]
    Each hyperplane is parallel to some linear hyperplane $x_i -x_j = 0$, which means $\pi(\mathcal{A}^{H_0})$ is a deformation of type $A$ Coxeter arrangement. On the other hand, for each pair $1 \le i < j \le n$ such that $i,j \ne l$, since $\mathcal{A}$ is non-degenerate,  there is at least one hyperplane of the form $x_i - x_j = a_{ij}$ for some $a_{ij} \in \mathbb{R}$ in $\mathcal{A}$ as well as in $\pi(\mathcal{A}^{H_0})$. Therefore, the arrangement $\pi(\mathcal{A}^{H_0})$ is non-degenerate. 
\end{proof}
Now we are ready to prove \Cref{thm:typeA}.
\begin{proof}[Proof of \Cref{thm:typeA}]
    We use induction on the dimension $n$ of the deformation $\mathcal{A}$.
    
    For the base case when $n = 2$, assume the deformation $\mathcal{A} = \{x_1-x_2 = a_1,\ldots,a_k\}$. It is easy to show that $r_1(\mathcal{A}) = k-1$ and $r_2(\mathcal{A}) = 2$. From the intersection poset of $\mathcal{A}$, we have the characteristic polynomial $\chi_{\mathcal{A}}(t) = t^2-kt = 2\binom{t}{2}-(k-1)\binom{t}{1}$.

    We assume that the theorem holds for all the deformation $\mathcal{A}$ in $\mathbb{R}^{n-1}$. For the case of dimension $n$, firstly we show that the Coxeter arrangement $\operatorname{Cox}_{A}(n)$ satisfies the expansion equation in the theorem. 
    Regions in $\operatorname{Cox}_{A}(n)$ have a natural bijection to permutations of length $n$. And each region has the level of $n$. Thus, 
    \[r(\operatorname{Cox}_{A}(n)) = r_n(\operatorname{Cox}_{A}(n)) = n!.\]
    On the other hand, by applying \Cref{thm:terao}, we know that
    \[\chi_{\operatorname{Cox}_{A}(n)}(t) = t\cdot(t-1)\cdots(t-n+1) = n! \binom{t}{n} = r_n(\operatorname{Cox}_{A}(n)) \binom{t}{n}.\]
    Next, we know that any deformation can be constructed step by step from $\operatorname{Cox}_{A}(n)$ (within finite steps), where each step involves either adding a hyperplane to or removing a hyperplane from the arrangement, while ensuring that the arrangement remains a non-degenerate deformation at each step. It is sufficient to show that the deformation preserves the expansion equation after adding or removing a hyperplane, given that the equation holds prior to the operation. Assume that $\mathcal{A}$ is any deformation satisfying the expansion, i.e.
    \[\chi_{\mathcal{A}}(t) = \sum_{k = 0}^{n} (-1)^{n-k} \cdot r_k(\mathcal{A}) \cdot \binom{t}{k}.\]
    Add a hyperplane $H_0$ to the arrangement $\mathcal{A}$ and denote the new deformation by $\tilde{\mathcal{A}}$, where $H_0$ is parallel to some hyperplane in $\mathcal{A}$. Denote the arrangement $\tilde{\mathcal{A}}^{H_0}$ by $\mathring{\mathcal{A}}$ for simplicity. By \Cref{lemma:nonde}, $\mathring{\mathcal{A}}$ is a non-degenerate deformation of type $A$ in $\mathbb{R}^{n-1}$. Applying the inductive hypothesis, the expansion holds for $\mathring{\mathcal{A}}$, i.e. 
    \[\chi_{\mathring{\mathcal{A}}}(t) = \sum_{k = 0}^{n-1} (-1)^{n-1-k} \cdot r_k(\mathring{\mathcal{A}}) \cdot \binom{t}{k}.\]
    $(\tilde{\mathcal{A}},\mathcal{A},\mathring{\mathcal{A}})$ is a triple of arrangements. Then by \Cref{lemma:d-r}, \[\chi_{\tilde{\mathcal{A}}}(t) = \chi_{\mathcal{A}}(t) - \chi_{\mathring{\mathcal{A}}}(t).\]
    Now we count the number of regions with each level that increases due to the addition of hyperplane $H_0$. Let $H \in \mathcal{A}$ be the nearest hyperplane parallel to $H_0$. All the increased regions appear between $H_0$ and $H$, which means they can not have the level of $n$. Meanwhile, the rest of the regions remain at the same level as before. Each newly-constructed region corresponds to a region of the arrangement $\tilde{\mathcal{A}}^{H_0}$ in $H_0$, which is exactly the intersection of the boundary of the region with $H_0$. See \Cref{ex:add} for an example, the regions increased are marked with stars. Since the distance between $H_0$ and $H$ is bounded, the vectors of each increased regions are constrained in the direction of the norm of $H_0$, which implies that each of the increased region has the same level as its corresponding region in $H_0$. Thus, 
    \[r_k(\tilde{\mathcal{A}}) = r_k(\mathcal{A}) + r_k(\mathring{\mathcal{A}})\]
    for each $1 \le k \le n-1$, and $r_n(\tilde{\mathcal{A}}) = r_n(\mathcal{A})$.
    We have
    \[\begin{aligned}
        \chi_{\tilde{\mathcal{A}}}(t)
        =& \chi_{\mathcal{A}}(t) - \chi_{\mathring{\mathcal{A}}}(t)\\
        =& \sum_{k = 0}^{n} (-1)^{n-k} \cdot r_k(\mathcal{A}) \cdot \binom{t}{k} - \sum_{k = 0}^{n-1} (-1)^{n-1-k} \cdot r_k(\mathring{\mathcal{A}}) \cdot \binom{t}{k}\\
        =& r_n(\mathcal{A})\cdot \binom{t}{n} + \sum_{k = 0}^{n-1} (-1)^{n-k} \cdot \left(r_k(\mathcal{A}) + r_k(\mathring{\mathcal{A}})\right) \cdot \binom{t}{k}\\
        =& \sum_{k = 0}^{n} (-1)^{n-k} \cdot r_k(\tilde{\mathcal{A}}) \cdot \binom{t}{k},
    \end{aligned}.\]
    The addition of a parallel hyperplane preserves the expansion equation. We use the same method to show that the removal of a parallel hyperplane preserves the expansion as well. Note that a hyperplane can only be removed if there is at least one other hyperplane parallel to it,in order to maintain its non-degenerate condition. These three arrangements described above can still form a triple of arrangements by swapping the order of the first two arrangements before and after the operation. The rest of the proof remains the same. Finally, we conclude that every step preserves the expansion and therefore all the non-degenerate deformations $\mathcal{A}$ in $\mathbb{R}^{n}$ satisfy the the expansion equation.
\end{proof}
\begin{figure}[h]
    \centering
    \begin{subfigure}{.4\textwidth}
    \centering
    \begin{tikzpicture}[scale = 0.6]
        \draw (-1,0)--(6,0);
        \draw (-1,2)--(6,2);
        \draw (2,-2)--(2,4);
        \draw (-1,4)--(5,-2);
    \end{tikzpicture}
    \caption{The arrangement $\mathcal{A}$.}
    \end{subfigure}
    \begin{subfigure}{.4\textwidth}
    \centering
    \begin{tikzpicture}[scale = 0.6]
        \draw (-1,0)--(6,0);
        \draw (-1,2)--(6,2);
        \draw (2,-2)--(2,4);
        \draw (-1,4)--(5,-2);
        \draw (1,4)--(7,-2);
        \node at (-1,4.5) {$H$};
        \node at (1,4.5) {$H_0$};
        \node at (1,3) {$*$};
        \node at (2.3,2.3) {$*$};
        \node at (3.25,1) {$*$};
        \node at (5,-1) {$*$};
    \end{tikzpicture}
    \caption{Adding $H_0$ in $\mathcal{A}$.}
    \end{subfigure}
    \caption{The regions increased after adding a hyperplane.}
    \label{ex:add}
\end{figure}
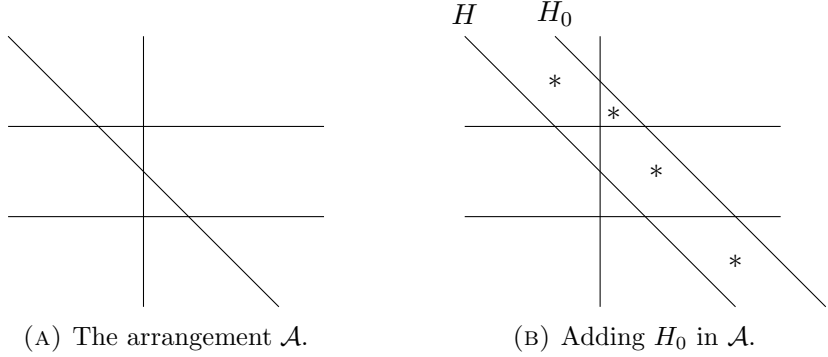

\subsection{Deformations of type $B$ Coxeter arrangements}
As in the previous subsection, we firstly present the analogous lemma for deformations of type $B$.
\begin{lemma}\label{lemma:nondeB}
    Let $\mathcal{B}$ be a non-degenerate deformation of type $B$ Coxeter arrangement in $\mathbb{R}^n$. Choose a hyperplane $H_0 \in \mathcal{B}$. Then the restriction of $\mathcal{B}$ on $H_0$ is a non-degenerate deformation of type $B$ Coxeter arrangement in $\mathbb{R}^{n-1}$. 
\end{lemma}

\begin{proof}
    We prove the lemma by examining the equation of $H_0$ case by case, using the same method as in the proof of \Cref{lemma:nonde}. Recall that each hyperplane $H \in \mathcal{A}$ has one of the following forms: $x_i = b$, $x_i - x_j = b$, or $x_i + x_j = b$ for some $1 \le i \ne j \le n$ and some $b \in \mathbb{R}$.
    
    \noindent\textbf{Case 1.} $H_0$ has the form
    \[H_0: x_k = a\]
    for some $1 \le k \le n$ and $a \in \mathbb{R}$. Let $\pi$ be a projection from $H_0$ to $\mathbb{R}^{n-1}$ by 
    \[\pi (x_1, \dots,x_k, \ldots x_n) = (x_1, \ldots,\hat{x}_k, \ldots,x_n).\]
    
    For any $H \in \mathcal{A} - H_0$ such that $H \cap H_0 \ne \emptyset$, there are three possibilities.
    
    If $H: x_i = b \ (i \ne k)$, then 
    \[\pi(H \cap H_0) = \{x \in \mathbb{R}^{n-1} \ | \ x_i = b\}.\]
    
    If $H: x_i -x_j = b$, then
    \[\pi(H \cap H_0) = \begin{cases}
        \{x \in \mathbb{R}^{n-1} \ | \ x_j = a-b\}, &\text{if}\ i = k, \\
        \{x \in \mathbb{R}^{n-1} \ | \ x_i = a+b\}, &\text{if}\ j = k, \\
        \{x \in \mathbb{R}^{n-1} \ | \ x_i-x_j = b\}, &\text{otherwise}.
    \end{cases}\]
    
    If $H: x_i +x_j = b$, then
    \[\pi(H \cap H_0) = \begin{cases}
        \{x \in \mathbb{R}^{n-1} \ | \ x_j = b-a\}, &\text{if}\ i = k, \\
        \{x \in \mathbb{R}^{n-1} \ | \ x_i = b-a\}, &\text{if}\ j = k, \\
        \{x \in \mathbb{R}^{n-1} \ | \ x_i+x_j = b\}, &\text{otherwise}.
    \end{cases}\]
    
    Those equations of hyperplanes $\pi(H \cap H_0)$ show that the arrangement $\mathcal{B}^{H_0} \cong \pi(\mathcal{B}^{H_0})$ is a deformation of type $B$ Coxeter arrangement. On the other hand, since $\mathcal{B}$ is non-degenerate, there is at least one hyperplane in each direction in $\mathcal{B}$. By iterating over all possible subscripts $i$, $j$ or $h$ of $H$, we obtain that the arrangement $\mathcal{B}^{H_0}$ is non-degenerate as well. 

    \
    
    \noindent\textbf{Case 2.} $H_0$ has the form
    \[H_0: x_k - x_l = a\]
    for some $1 \le k < l \le n$ and $a \in \mathbb{R}$. Let $\pi$ be a projection from $H_0$ to $\mathbb{R}^{n-1}$ by 
    \[\pi (x_1, \dots,x_l, \ldots x_n) = (x_1, \ldots,\hat{x}_l, \ldots,x_n).\]
    
    For any $H \in \mathcal{A} - H_0$ such that $H \cap H_0 \ne \emptyset$, there are three possibilities. 
    
    If $H: x_h = b$, then 
    \[\pi(H \cap H_0) = \begin{cases}
        x_k = a+b, &\text{if}\ h = l,\\
        x_h = b, &\text{if}\ h \ne l.
    \end{cases}\]
    
    If $H: x_i -x_j = b$, then $\pi(H \cap H_0)$ is shown as in the proof of \Cref{lemma:nonde}.
    
    If $H: x_j +x_j = b$, then
    \[\pi(H \cap H_0) = \begin{cases}
        x_k + x_j = a+b , &\text{if}\ i = l, \\
        x_k + x_i = a+b , &\text{if}\ j = l, \\
        x_i + x_j = b, &\text{otherwise}.\\
    \end{cases}\]
    
    As the argument above, we can conclude that the arrangement $\mathcal{B}^{H_0} \cong \pi(\mathcal{B}^{H_0})$ is a non-degenerate deformation of type $B$ Coxeter arrangement in $\mathbb{R}^{n-1}$. 

    \
    
    \noindent\textbf{Case 3.} $H_0$ has the form
    \[H_0: x_k + x_l = a\]
    for some $1 \le k < l \le n$ and $a \in \mathbb{R}$. The proof is similar to Case 2, so we will not repeat it here.
\end{proof}

\begin{proof}[Proof of \Cref{thm:typeB}]
    We prove the theorem by induction on dimension $n$ of the deformation $\mathcal{B}$.
    For the base case when $n = 1$, assume the deformation $\mathcal{B} = \{x_1 = a_1,\ldots,a_k\}$. Then $r_0(\mathcal{B}) = k-1$ and $r_1(\mathcal{B}) = 2$. We have the characteristic polynomial
    \[\chi_{\mathcal{B}}(t) = t - kt = 2\binom{\frac{t-1}{2}}{1} - (k-1)\binom{\frac{t-1}{2}}{0},\]
    satisfying the expansion of the theorem.
    
    For the Coxeter arrangement $\operatorname{Cox}_{B}(n)$ of any dimension $n$, there is a natural bijection between regions in $\operatorname{Cox}_{B}(n)$ and signed permutations of length $n$. And each region has the level of $n$. Thus,
    \[r(\operatorname{Cox}_{B}(n)) = r_n(\operatorname{Cox}_{B}(n)) = 2^n \cdot n!.\]
    Moreover, by \Cref{thm:terao}, we have
    \[\chi_{\operatorname{Cox}_{B}(n)}(t) = (t-1) \cdot (t-3) \cdots (t-(2n-1)) = 2^n\cdot n! \cdot \binom{\frac{t-1}{2}}{n} = r_n(\operatorname{Cox}_{B}(n))\binom{\frac{t-1}{2}}{n},\]
    which implies that the expansion equation holds for Coxeter arrangements of type $B$.
    
    The subsequent steps are the same as the arguments in the proof of \Cref{thm:typeA}. Any deformations of dimension $n$ can be constructed step by step from the Coxeter arrangement $\operatorname{Cox}_{B}(n)$, with each step adding a hyperlane to or removing a hyperplane from the arrangement. By \Cref{lemma:nondeB}, the restriction preserves the properties of non-degenerate deformations. Therefore, we can apply the deletion-restriction lemma combined with the inductive hypothesis of lower dimension to conclude the result recursively. 
\end{proof}

\subsection{Applications of main results}
\Cref {thm:typeA} recovers and generalizes recent results on several specific types of arrangements, while \Cref{thm:typeB} further extends \Cref{thm:typeA} to type $B$.

The Catalan-type arrangement is defined by
\[\mathcal{C}_{n,A} = \{x_i -x_j = 0,\pm a_1,\pm a_2,\ldots,\pm a_m\ |\ 1 \le i \ne j \le n\},\]
where $A = \{a_1 > a_2 > \ldots > a_m\}$ is a positive real number set. Note that when $A = \{1\}$ the arrangement is the classical Catalan arrangement and when $A = [m]$ it becomes the $m$-Catalan arrangement.

The semiorder-type arrangement is defined by 
\[\mathcal{C}^*_{n,A} = \{x_i -x_j = \pm a_1,\pm a_2,\ldots,\pm a_m\ |\ 1 \le i \ne j \le n\}\]
where $A = \{a_1 > a_2 > \ldots > a_m\}$ is a positive real number set.

Very recently, Chen et al. \cite{wsj,wsj2} established the following result on the characteristic polynomials of a specific type of arrangements.
\begin{theorem}[\cite{wsj,wsj2}]\label{thm:wsj_chi}
    Let $\bar{\mathcal{A}}_n = \{x_i - x_j = a_1, a_2,\ldots, a_m\ |\ 1 \le i \ne j \le n\}$, where $A = \{a_1,a_2,\ldots,a_m\}$ is a real number set. Then
    \[\chi_{\bar{\mathcal{A}}_n}(t) = \sum_{k = 0}^n (-1)^{n-k}r_{k}(\mathcal{A}_n) \binom{t}{k}.\]
    In particular, the formula holds for both $\mathcal{C}_{n,A}$ and $\mathcal{C}^*_{n,A}$.
\end{theorem}

Note that the arrangements mentioned above are non-degenerate deformations of type $A$ Coxeter arrangement. Our main result \Cref{thm:typeA} not only recovers \Cref{thm:wsj_chi}, but also extends the above result to \textbf{any} non-degenerate deformations of the Braid arrangement through much simpler arguments. While the proof of \Cref{thm:wsj_chi} relies on certain symmetries of coefficients, our results \Cref{thm:typeA} and \Cref{thm:typeB} employs a more general method and requires much fewer ristrictions on arrangements.

\

Moreover, our main results \Cref{thm:typeA} and \Cref{thm:typeB} provide a new approach to determine the characteristic polynomial of hyperplane arrangement.

The authors of \cite{wsj} provided the following enumeration result for the regions with fixed levels in $m$-Catalan arrangements. 

\begin{theorem}[\cite{wsj}]
    The number of regions $r_k(\mathcal{C}_{n,[m]})$ with level $k$ in $m$-Catalan arrangement $\mathcal{C}_{n,[m]}$ is given by
    \[r_k(\mathcal{C}_{n,[m]}) = \frac{n!mk}{(m+1)n{-}k}\binom{(m+1)n{-}k}{mn}.\]
\end{theorem}

Combined with \Cref{thm:typeA},we immediately obtain the characteristic polynomial for the $m$-Catalan arrangement.

Many hyperplane arrangements of interest to combinatorialists are non-degenerate deformations of Coxeter arrangements. By counting the number of regions with fixed level for more such arrangements, \Cref{thm:typeA} and \Cref{thm:typeB} can be applied to derive their characteristic polynomials in a novel way. 

\section*{Acknowledgements}
The author thanks their advisor Prof. Yibo Gao for proposing this research direction and for his invaluable guidance and feedback.

\bibliographystyle{plain}
\bibliography{ref.bib}
\end{document}